\documentclass[12pt
]{article}

\usepackage{amssymb}
\usepackage{amsmath}
\usepackage{mathrsfs}
\usepackage{theorem}

\usepackage{newtxtext}

\newtheorem{theorem}{Theorem}
\newtheorem{lemma}{Lemma}
\newtheorem{corollary}{Corollary}
\newenvironment{proof} {{\bf Proof.}}{\hfill \fbox{}\\ \smallskip}

{\theorembodyfont{\rmfamily} \newtheorem{remark}{Remark}}
{\theorembodyfont{\rmfamily} }
{\theorembodyfont{\rmfamily} }

\newcommand{\C}{\mathbb{C}}

\newcommand{\h}{\eta}

\newcommand{\de}{\partial}
\newcommand{\ga}{\gamma}
\newcommand{\ep}{\varepsilon}

\newcommand{\vf}{\varphi}

\newcommand{\p}{\psi}

\newcommand{\si}{\sigma}

\newcommand{\al}{\alpha}
\newcommand{\be}{\beta}

\newcommand{\la}{\lambda}
\newcommand{\La}{\Lambda}

\newcommand{\Om}{\Omega}
\newcommand{\om}{\omega}

\newcommand{\lan}{\langle}
\newcommand{\ran}{\rangle}

\newcommand{\essinf}{\mathop{\rm ess\, inf}}
\newcommand{\esssup}{\mathop{\rm ess\, sup}}
\newcommand{\tr}{\operatorname{\rm tr}}

\renewcommand\leq{\leqslant}
\renewcommand\geq{\geqslant}

\newcommand{\R}{\mathbb R}

\newcommand\Cspt{\mathaccent"017{C}}
\newcommand{\Hspt}{\mathaccent"017{H}}

\renewcommand\Re{\mathop{\mathbb R \rm{e}}\nolimits}

\newcommand\A{\mathop{\mathscr A}\nolimits}

\newcommand\dist{\operatorname{dist}}
\newcommand\spt{\text{spt}\,}

\title{The functional dissipativity of certain systems of partial differential equations}

\author{A. Cialdea
\thanks{Department of Mathematics, Computer Sciences and Economics,
University of Ba\-si\-li\-ca\-ta, V.le dell'Ateneo Lucano, 10, 85100 Potenza, Italy.
 \textit{email:}
cialdea@email.it.}\and
V. Maz'ya 
\thanks{Department of Mathematics, Link\"oping University,
SE-581 83, Link\"oping, Sweden.
\textit{email}: vladimir.mazya@liu.se.}
}
\date{}    

\begin{document}

\maketitle

\begin{abstract}
In the present paper we consider the functional dissipativity of 
the Dirichlet problem for 
systems of partial differential
operators of the form $\de_{h} (\A^{hk}(x)\de_{k})$ ($\A^{hk}$ being $m\times m$ matrices with
complex valued $L^{1}_{\text{loc}}$ entries).  
In the particular case of the operator $ \de_{h} (\A^{h}(x)\de_{h})$ (where $\A^{h}$ are 
$m\times m$ matrices) we obtain 
algebraic necessary and sufficient conditions. 
We give also three different notions of functional ellipticity
and investigate  the relations between them and the functional dissipativity
for the operators in question.
\end{abstract}

 \section{Introduction}
 

 The concept of $L^{\Phi}$-dissipativity (or functional dissipativity) of a linear operator
 was recently introduced in  \cite{CM2021}. In particular, if $A$ is the scalar second order
 partial differential operator $\nabla(\A \nabla)$, $\A$ being a square matrix whose entries are complex
 valued $L^{\infty}$-functions defined in a domain $\Om\subset \R^{n}$, we say that $A$
 is $L^{\Phi}$-dissipative with respect to a given positive function $\vf: \R^{+}\to \R^{+}$ if
 $$
 \Re \int_{\Om} \lan \A \nabla u, \nabla(\vf(|u|)\, u)\ran\, dx \geq 0
 $$
 for any $u\in\Hspt^{1}(\Om)$ such that $\vf(|u|)\, u \in \Hspt^{1}(\Om)$.
 
 We say that $A$
 is strict $L^{\Phi}$-dissipative if there exists $\kappa>0$ such that
 $$
 \Re \int_{\Om} \lan \A \nabla u, \nabla(\vf(|u|)\, u)\ran\, dx \geq \kappa \int_{\Om} 
 |\nabla(\sqrt{\vf(|u|)}\, u)|^{2}dx 
 $$
 for any $u\in\Hspt^{1}(\Om)$ such that $\vf(|u|)\, u \in \Hspt^{1}(\Om)$.

The functional dissipativity is an extension of the concept of $L^{p}$-dissipativity, which 
is obtained taking $\vf(t)=t^{p-2}$ ($1<p<\infty$).
In a series of papers   \cite{CM2005,CM2006,CM2013,CM2018}
we have studied the problem of characterizing the $L^{p}$-dissipativity of scalar and
matrix partial differential operators. A recent survey can be found in \cite{CM20222}.
 In the monograph \cite{CMbook}  these results
are considered in the more general frame of the theory of semi-bounded operators.

Later the concept of $p$-ellipticity, which is closely related to the strict $L^{p}$-dissipativity,  has been introduced. It has been considered in
different papers by  Carbonaro and Dragi\v{c}evi\'{c} \cite{CD20201,CD20202},  Dindo\v{s} and Pipher \cite{DP20191,DP20192,DP20201,DP20202},
Egert \cite{egert}. 
It is worthwhile to remark that, if the partial differential operator has no lower order terms, the concepts of $p$-ellipticity
and strict  $L^p$-dissipativity coincide. 

Recently Dindo\v{s}, Li and Pipher \cite{DLP} have considered the $p$-ellipticity for
linear systems of partial differential equations, giving three different definitions: the strong $p$-ellipticity, the integral  $p$-ellipticity
and the weak $p$-ellipticity, showing that strong $p$-ellipticity $\Rightarrow$  integral  $p$-ellipticity
$\Rightarrow$ weak $p$-ellipticity.
 
 The aim of the present paper is to study the $L^{\Phi}$-dissipativity of the  operator
 $$
 Au = \de_{h}(\A^{h}(x)\de_{h}u)
 $$
 where $\A^{h}(x)=\{a_{ij}^{h}(x)\}$  ($i,j=1,\ldots,m$)
are matrices
with complex locally integrable entries defined in a domain 
$\Om\subset\R^{n}$ ($h=1,\ldots,n$).
We give necessary and sufficient condition for the functional dissipativity of such
an  operator
and
also for its functional ellipticity, a concept which was introduced in \cite{CM2021}
for scalar operators and generalizes the concept of $p$-ellipticity.
As done in  \cite{DLP} for $p$-ellipticity,  here we consider three kinds of functional ellipticity for
the operator $A$: the strong, the integral and the weak functional ellipticity. We prove  that all of them, but the strong one,  are equivalent to the
strict $L^{\Phi}$-dissipativity of $A$. We prove also that the integral
functional ellipticity of this operator takes place if and only if there exists $\kappa>0$ such that
$A-\kappa\Delta$ is $L^{\Phi}$-dissipative. These results appear  to be interesting
even in the particular case of $p$-ellipticity.

The present paper is organized as follows. 
After some preliminaries in Section \ref{sec:prel}, we 
study the $L^\Phi$-dissipativity for systems of ordinary differential
equations in Section \ref{sec:ODE}. 
The particular case of real operators is considered in Section 
 \ref{sec:real}, where several criteria are given in terms of eigenvalues
 of the coefficient matrices.
 
 Section \ref{sec:PDE} is devoted to the $L^{\Phi}$-dissipativity
 of operator $A$.  Finally, in the last Section \ref{sec:ellipt} 
some concepts of functional ellipticity for systems are introduced and 
 the relations between them are investigated for the operators 
 considered in the paper.

\section{Preliminaries}\label{sec:prel}

The positive function $\vf$ is required to satisfy the following conditions

\renewcommand{\labelenumi}{(\roman{enumi})}
\renewcommand{\theenumi}{(\roman{enumi})}
\begin{enumerate}
	\item\label{item1} $\vf \in C^{1}((0,+\infty))$;
	\item $(s\, \vf(s))'>0$ for any $s>0$;
	\item the range of the strictly increasing function $s\, \vf(s)$ is  $(0,+\infty)$;	
	\item\label{item4} there exist two positive constants $C_{1}, C_{2}$  and a real number $r>-1$ such that
$$
C_{1} s^{r}\leq (s\vf(s))' \leq C_{2}\, s^{r}, \qquad s\in (0,s_{0})
$$
for a certain $s_{0}>0$. If $r=0$ we require more restrictive 
conditions: there exists the finite limit $\lim_{s\to 
0^+}\vf(s)=\vf_{+}(0)>0$ 
and  $\lim_{s\to 0^+}s\, \vf'(s)=0$.
\item\label{item5} 
There exists $s_{1}>s_{0}$ such that 
$$
   \vf'(s)\geq 0  \text{ or }  \vf'(s)\leq 0 \qquad \forall\ s\geq s_{1}
   . 
$$
\end{enumerate}

 We note that from condition  \ref{item4} it follows that, for any $r>-1$, 
	$$
		\vf(s)  \simeq s^{r}, \qquad s\in (0,s_{0}).
$$

Let us denote by $t\, \psi(t)$ the inverse function of $s\, \vf(s)$. 
The functions
$$
\Phi(s)=\int_{0}^{s} \si\, \vf(\si)\, d\si, \qquad
\Psi(s)= \int_{0}^{s} \si\, \psi(\si)\, d\si
$$
are conjugate Young functions. Moreover
the function $\vf$ satisfies conditions \ref{item1}-\ref{item5}
if and only if the function $\psi$ satisfies
the same conditions with $-r/(r+1)$ instead of $r$
(see \cite[Lemma 1]{CM2021}).

In the following we shall use the function $\La$ which is defined by
the relation 
$$\Lambda\left(s\sqrt{\vf(s)}\right)= - \frac{s\, \vf'(s)}{s\,\vf'(s)+2\, 
\vf(s)}\, .
$$

Let us consider a  general system of the form
\begin{equation}
    A=\de_{h} (\A^{hk}(x)\de_{k})
    \label{eq:A}
\end{equation}
where  $\de_{k} =\de / \de x_{k}$ and $\A^{hk}(x)=\{a^{hk}_{ij}(x)\}$ are 
$m\times m$ matrices 
whose elements are complex valued $L^{1}_{\text{loc}}$-functions  defined 
in a domain $\Om\subset \R^n$ 
$(1\leq i,j\leq 
m,\ 1\leq h,k \leq n)$.  Here and in the sequel, we adopt the standard summation convention
on repeated indices. We say that the operator \eqref{eq:A} is $L^\Phi$-dissipative if
\begin{equation}\label{eq:defdiss0}
\Re \int_{\Om} \lan \A^{hk} \de_{k} u, \de_{h}(\vf(|u|)\, u)\ran\, dx \geq 0
\end{equation}
for any $u\in [\Cspt^{1}(\Om)]^m$  such that $\vf(|u|)\, u\in 
[\Cspt^{1}(\Om)]^m$.

We say that the operator \eqref{eq:A} is strict $L^\Phi$-dissipative if
there exists $\kappa>0$ such that
\begin{equation}\label{eq:defdisstrict}
\Re \int_{\Om} \lan \A^{hk} \de_{k} u, \de_{h}(\vf(|u|)\, u)\ran\, dx \geq 
\kappa \int_{\Om} |\nabla(\sqrt{\vf(|u|)}\, u)|^{2}dx
\end{equation}
for any $u\in [\Cspt^{1}(\Om)]^m$  such that $\vf(|u|)\, u\in 
[\Cspt^{1}(\Om)]^m$.

\begin{remark}
The coefficients of the operator \eqref{eq:A} are
supposed to be $L^{1}_{\text{loc}}$-functions  and not
 $L^{\infty}$-functions, as in \cite{CM2022}. This is the reason why 
 in definitions \eqref{eq:defdiss0} and \eqref{eq:defdisstrict} we consider  
 the space $[\Cspt^{1}(\Om)]^m$ instead of
$[\Hspt^{1}(\Om)]^m$, as done in  \cite{CM2022}.
\end{remark}

\begin{remark}\label{rem:1} 
If $r\geq 0$ in condition \ref{item4}, $u\in [\Cspt^{1}(\Om)]^m$  implies $\vf(|u|)\, u\in 
[\Cspt^{1}(\Om)]^m$ and then  we can say that the operator \eqref{eq:A} is $L^\Phi$-dissipative if
and only if \eqref{eq:defdiss0} holds for any $u\in [\Cspt^{1}(\Om)]^m$.
If $r<0$, setting $w=\vf(|u|)\, u$, i.e. $u=\psi(|w|)\, w$, we  can write  condition \eqref{eq:defdiss0} as
\begin{equation}\label{eq:rewrite}
\Re \int_{\Om} \lan (\A^{kh})^{*} \de_{k} w, \de_{h}(\psi(|w|)\, w)\ran\, dx \geq 0
\end{equation}
for any $w\in [\Cspt^{1}(\Om)]^{m}$ such that $\psi(|w|)\, w\in [\Cspt^{1}(\Om)]^{m}$.
Thanks to \cite[Lemma 1]{CM2021}, the function $\psi$ satisfies
 conditions \ref{item1}-\ref{item5} with $-r/(r+1)$ instead of $r$. The number
 $-r/(r+1)$ being greater than $0$, 
 the operator \eqref{eq:A} is $L^\Phi$-dissipative if
and only if \eqref{eq:rewrite} holds for any $w\in [\Cspt^{1}(\Om)]^m$.
The same remark applies to \eqref{eq:defdisstrict}, since 
$\sqrt{\psi(|w|)}\, w = \sqrt{\vf(|u|)}\, u$ (see \cite[formula (43)]{CM2021}).
\end{remark}

We end this Section by proving a Lemma similar to \cite[Lemma 6]{CM2022}.

\begin{lemma}\label{le:lemma1}
Let $\Om$ be  a domain in $\R^n$. The operator \eqref{eq:A} is $L^{\Phi}$-dissipative if and only if
\begin{equation}\label{eq:cond1}
\begin{gathered}
\Re \int_{\Om} \Big(\lan \A^{hk} \de_k	 v, \de_h v\ran 
+\La(|v|)\, |v|^{-2} \lan \left(\A^{hk}- (\A^{kh})^*\right)v, \de_h v \ran \Re \lan v, \de_{k} v\ran  \\
- \La^{2}(|v|)\, |v|^{-4} \lan \A^{hk} v,v\ran \Re \lan v, \de_{k} v\ran \Re \lan v, \de_{h} v\ran
\Big) dx  \geq 0
\end{gathered}
\end{equation}
for any $v \in [\Cspt^{1}(\Om)]^m$. Here and in the sequel the integrand
is extended by zero on the set where $v$ vanishes.
\end{lemma}
\begin{proof}
{\it Sufficiency.}
First suppose $r\geq 0$ in condition \ref{item4}. Take $u\in  [\Cspt^1(\Om)]^m$ and define $v=\sqrt{\vf(|u|)}\, u$. Also $v$ belongs to 
$[\Cspt^1(\Om)]^m$  and, reasoning as in the proof of \cite[Lemma 6]{CM2022}, we find
the identity 
\begin{gather*}
\Re \lan \A^{hk} \de_{k} u, \de_{h}(\vf(|u|)\, u)\ran =\Re \big(
\lan \A^{hk} \de_{k} v, \de_{h} v\ran \\
+ \La(|v|) |v|^{-2} 
\lan (\A^{hk}-(\A^{kh})^{*})v, \de_{h} v\ran \Re \lan v, \de_{k} v\ran \\
	 -\La^{2}(|v|)|v|^{-4} \lan \A^{hk}v, v\ran \, \Re \lan v, \de_{h} v\ran
	 \Re \lan v, \de_{k} v\ran \big)
\end{gather*}
on the set $\{x\in \Om\ |\ u(x)\neq 0\}=\{x\in \Om\ |\ v(x)\neq 0\}$.
Inequality \eqref{eq:cond1} implies \eqref{eq:defdiss0} and the sufficiency is proved
when $r\geq 0$.

If $-1<r<0$, 
recalling Remark \ref{rem:1}, we can say that operator \eqref{eq:A} is $L^{\Phi}$-dissipative
if and only if \eqref{eq:rewrite} holds for any $w\in [\Cspt^{1}(\Om)]^{m}$.
 Therefore  what we have already proved for $r\geq 0$ shows that \eqref{eq:rewrite}  holds 
 if 
 \begin{equation}\label{eq:condvstar}
\begin{gathered}
\Re \int_{\Om} \Big(\lan (\A^{kh})^{*} \de_k	 v, \de_h v\ran 
+\widetilde{\La}(|v|)\, |v|^{-2} \lan \left((\A^{kh})^{*}- \A^{hk}\right)v, \de_h v \ran \Re \lan v, \de_{k} v\ran  \\
-\widetilde{\La}^{2}(|v|)\, |v|^{-4} \lan (\A^{kh})^{*} v,v\ran \Re \lan v, \de_{k} v\ran \Re \lan v, \de_{h} v\ran
\Big) dx  \geq 0
\end{gathered}
\end{equation}
for any $v \in [\Cspt^{1}(\Om)]^m$, where $\widetilde{\La}$  is defined by
the relation 
$$\widetilde{\La}\left(s\sqrt{\psi(s)}\right)= - \frac{s\, \psi'(s)}{s\,\psi'(s)+2\, 
\psi(s)}\, .
$$
Since $\widetilde{\Lambda}(|v|)=-\Lambda(|v|)$ (see \cite[Lemma 2]{CM2021}), condition \eqref{eq:condvstar}
coincides with \eqref{eq:cond1} and the sufficiency is proved also for $-1<r<0$.

{\it Necessity.} 
If $r\geq 0$ we can repeat the first part of the proof of Necessity in  \cite[Lemma 6]{CM2022}
to show that \eqref{eq:defdiss0} implies \eqref{eq:cond1} for any $v \in [\Cspt^{1}(\Om)]^m$.
If $r<0$ we rewrite \eqref{eq:defdiss0} as \eqref{eq:rewrite}. As before, what we have proved
for $r\geq 0$ shows that \eqref{eq:rewrite} implies \eqref{eq:condvstar} for any $v \in [\Cspt^{1}(\Om)]^m$ and this concludes the proof. 
\end{proof}

In the same way (see also \cite[Lemma 7]{CM2022}) one can prove the next result
\begin{lemma}\label{le:lemma2}
Let $\Om$ be  a domain in $\R^n$. The operator \eqref{eq:A} is strict $L^{\Phi}$-dissipative if and only if
there exists $\kappa>0$ such that
\begin{equation}\label{eq:cond1str}
\begin{gathered}
\Re \int_{\Om} \Big(\lan \A^{hk} \de_k	 v, \de_h v\ran 
+\La(|v|)\, |v|^{-2} \lan \left(\A^{hk}- (\A^{kh})^*\right)v, \de_h v \ran \Re \lan v, \de_{k} v\ran  \\
- \La^{2}(|v|)\, |v|^{-4} \lan \A^{hk} v,v\ran \Re \lan v, \de_{k} v\ran \Re \lan v, \de_{h} v\ran
\Big) dx  \geq \kappa \int_{\Om} |\nabla v|^{2}dx
\end{gathered}
\end{equation}
for any $v \in [\Cspt^{1}(\Om)]^m$. 
\end{lemma}

 \section{Functional dissipativity for systems of ordinary differential equations}
 \label{sec:ODE}

In this Section we are going to consider the operator $A$ defined as
\begin{equation} \label{eq:Aord}
    A u = (\A(x)u')'
\end{equation}
where $\A(x)=\{a_{ij}(x)\}$ ($i,j=1,\ldots,m$) is a matrix with 
complex locally integrable entries defined in the bounded or 
unbounded
interval $(a,b)\subset \R$.

 \begin{lemma}\label{lem:1}
The operator $A$ is $L^{\Phi}$-dissipative if and only if
\begin{equation}\label{eq:condord1}
\begin{gathered}
  \int_{a}^{b}\Big( \Re \lan \A v',v'\ran-\La^{2}(|v|)\, |v|^{-4} \Re \lan \A v,v\ran 
	  (\Re \lan v,v'\ran)^{2} \cr
	  +\La(|v|)\, |v|^{-2} \Re (\lan \A v,v'\ran 
	  -\lan \A v',v\ran) \Re \lan v,v'\ran
	  \Big)\, dx  \geq 0,
\end{gathered}
\end{equation}
for any $v\in [\Cspt^1((a,b))]^m$.
\end{lemma}
\begin{proof}
This is just a particular case of Lemma \ref{le:lemma1}.
\end{proof}

As in \cite{CM2022}, from now on we require also the following condition on the function $\vf$:
\begin{enumerate}\setcounter{enumi}{5}
\item  the function
$$
|s\, \vf'(s)/ \vf(s)|
$$
is not decreasing.
\end{enumerate}

This condition implies that the function $\La^{2}(t)$ is not decreasing on $(0,+\infty)$
(see \cite[Lemma 8]{CM2022}).

 The next Theorem provides an algebraic necessary and sufficient
 condition for the $L^{\Phi}$-dissipative of operator \eqref{eq:Aord}.
 
 \begin{theorem}\label{th:1}
    The operator $A$ is $L^{\Phi}$-dissipative if and only if
    \begin{equation}\label{eq:condode}
    \begin{gathered}
  \Re \lan \A(x) \la,\la\ran - \La_{\infty}^{2}\Re\lan \A(x)\om,\om\ran (\Re \lan\la,\om\ran)^{2}
	\cr
	+
	\La_{\infty}\Re(\lan \A(x)\om,\la\ran -\lan \A(x)\la,\om\ran)
	\Re \lan \la,\om\ran   \geq 0
    \end{gathered}
\end{equation}
    for almost every $x\in(a,b)$ and for any 
    $\la,\om\in\C^{m}$, $|\om|=1$.
\end{theorem}
\begin{proof}
 \textit{Necessity.}
   Assume for the moment that
   the coefficients $a_{ij}$ are constant and  $(a,b)=\R$.

    Let us fix $\la$ and $\om$ in $\C^{m}$, with $|\om|=1$, and choose
    $v(x)=\h(x/R)\, w(x)$ where
    $$
    w_{j}(x)=\begin{cases}
    \mu\,\om_{j} & \text{if $x<0$'}\cr
    \mu\,\om_{j}+x^{2}(3-2x)\la_{j}   & \text{if $0\leq x\leq 1$,}\cr
    \mu\,\om_{j}+\la_{j}    & \text{if $x>1$},
    \end{cases}
    $$
    where $\mu,R\in\R^{+}$, $\h\in\Cspt^{\infty}(\R)$, $\spt\h\subset 
    [-1,1]$ and $\h(x)=1$ if $|x|\leq 1/2$.

Put this $v$ in \eqref{eq:condord1}. By repeating the arguments used in \cite[pp.250--251]{CM2006}
and observing that
$$
\La(|w(x)|)=
\begin{cases}
\La(\mu)  &  \text{if $x<0$,}\\
 \La(|\mu\om + x^2(3-2x)\la|) &  \text{if $0\leq x\leq 1$,}\\
 \La(|\mu\om + \la| & \text{if $x>1$},
\end{cases}
$$
implies $\lim_{\mu\to \infty}\La(|w(x)|)=\La_{\infty}$, we find 
 \begin{gather*}
36 \int_{0}^{1}\Big( \Re \lan \A \la,\la\ran - \La_{\infty}^{2}\Re\lan \A\om,\om\ran (\Re \lan\la,\om\ran)^{2}
	\cr
	+
	\La_{\infty}\Re(\lan \A\om,\la\ran -\lan \A\la,\om\ran)
	\Re \lan \la,\om\ran
		     \Big)\, x^{2}(1-x)^{2} dx  \geq 0
		     \end{gather*}
    and \eqref{eq:condode} is proved.

If $\{a_{hk}\}$ are defined on $(a,b)$ and are not necessarily constant, consider 
$$
v(x)=\psi((x-x_{0})/\ep)
$$
where $x_{0}$ is a fixed point in $(a,b)$, $\psi\in [\Cspt^{1}((-1,1))]^{m}$
and $0<\ep <  \min\{x_{0}-a,b-x_{0}\}$.
Putting $v$ in \eqref{eq:condord1} we get
\begin{gather*}
 \int_{\R}\Big(\Re \lan \A (x_0+\ep y)\p',\p'\ran 
	   -\La_{\infty}^{2}|\p|^{-4}\Re\lan \A (x_0+\ep y)\psi,\psi\ran (\Re \lan 
	   \psi,\p'\ran)^{2} \cr
	   + \La_{\infty} |\p|^{-2}\Re(\lan \A (x_0+\ep y)\psi,\p'\ran
   -\lan \A (x_0+\ep y)\p',\psi\ran) \Re \lan \psi,\p'\ran
	   \Big)\, dy  \geq 0.
	   \end{gather*}

 Letting $\ep\to 0^{+}$ we find for almost every $x_{0}$
	   \begin{gather*}
	   		   \int_{\R}\Big(\Re \lan \A (x_0)\p',\p'\ran 
		   -\La_{\infty}^{2}|\p|^{-4}\Re\lan \A (x_0)\psi,\psi\ran (\Re \lan 
		   \psi,\p'\ran)^{2} \cr
		   + \La_{\infty} |\p|^{-2}\Re(\lan \A (x_0)\psi,\p'\ran
	   -\lan \A (x_0)\p',\psi\ran) \Re \lan \psi,\p'\ran
		   \Big)\, dy  \geq 0.   
		  \end{gather*}

What we have obtained 
   for constant 
    coefficients gives the result.

     \textit{Sufficiency.}
     We start by observing that    \eqref{eq:condode}  implies
        \begin{equation}\label{eq:posit}
 \Re \lan \A(x) \la,\la\ran \geq 0
\end{equation}
for almost every $x\in(a,b)$ and for any 
    $\la\in\C^{m}$. 
    Indeed, let us take $x\in(a,b)$ such that
     \eqref{eq:condode} holds and $\la\in\C^{m}$; 
     we can choose $\om\in\C^{m}$, $|\om|=1$, such that $\lan\la,\om\ran=0$,
     and \eqref{eq:posit} is proved. 
  
     Let us fix $x\in(a,b)$ and
    $\la,\om\in\C^{m}$, $|\om|=1$, such that \eqref{eq:condode} holds.
 We can write this condition as
\begin{equation}\label{eq:condabc}
\al \La_{\infty}^{2} - \be \La_{\infty} - \ga \leq 0\, ,
\end{equation} 
 where
\begin{gather*}
 \al = \Re\lan \A(x)\om,\om\ran (\Re \lan\la,\om\ran)^{2}, \quad
 \ga = \Re \lan \A(x) \la,\la\ran, \\
 \be = \Re(\lan \A(x)\om,\la\ran -\lan \A(x)\la,\om\ran)\Re \lan\la,\om\ran.
\end{gather*}

Note that, in view of  \eqref{eq:posit}, $\al$ and $\ga$ are non-negative.
Then, if $\al>0$, we have $\La_1\, \La_2 \leq 0$, where $\La_1$ and $\La_2$
are the roots of the equation 
$\al \La^{2} - \be \La - \ga = 0$.
 We claim that from \eqref{eq:condabc} 
it follows
\begin{equation}\label{eq:condabct}
\al \La^{2}(t) - \be \La(t) - \ga \leq 0
\end{equation}
for any $t\geq 0$.
Indeed, suppose $\al>0$ and  $\La_1\leq 0 \leq \La_2$. Inequality \eqref{eq:condabc} means
\begin{equation*}
\La_1\leq  \La_{\infty} \leq \La_2\, .
\end{equation*}
We recall that the function $\La(t)$ is monotone and does not change sign (see \cite[Lemma 8]{CM2022}).
If $\La(t)\geq 0$, we have
$$
0\leq \La(t) \leq \La_{\infty} \leq \La_2,
$$
while if $\La(t)\leq 0$,
$$
\La_1 \leq \La_{\infty}\leq \La(t) \leq 0.
$$
In any case, we find $\La_1 \leq \La(t) \leq \La_2$ for any $t\geq 0$ and
\eqref{eq:condabct} is proved under the assumption  $\al>0$.

If  $\al=0$ and $\be=0$ \eqref{eq:condabct} is trivial. If $\al=0$ and $\be\neq 0$,
\eqref{eq:condabc} becomes $\be \La_{\infty} + \ga \geq 0$. By using  again the fact that
$\La(t)$ does not change sign, we get $\be \La(t) + \ga \geq 0$ for any $t\geq 0$.
We have then  proved that
 \begin{equation}\label{eq:condodet}
    \begin{gathered}
  \Re \lan \A(x) \la,\la\ran -\La^{2}(t)\Re\lan \A(x)\om,\om\ran (\Re \lan\la,\om\ran)^{2}
	\cr
	+
	\La(t)\Re(\lan \A(x)\om,\la\ran -\lan \A(x)\la,\om\ran)
	\Re \lan \la,\om\ran   \geq 0
    \end{gathered}
\end{equation}
 for almost every $x\in(a,b)$ and for any $t > 0$, 
    $\la,\om\in\C^{m}$, $|\om|=1$.
    This shows that the integrand in \eqref{eq:condord1} is non-negative almost everywhere 
    and Lemma \ref{lem:1} gives the result.
\end{proof}

 \begin{remark}\label{rem:equiv}
  In the proof  of the Sufficiency we have shown that - assuming condition \eqref{eq:condL} -
  inequality  \eqref{eq:condode} implies that \eqref{eq:condodet} holds for any $t>0$. The viceversa being obvious, we have that \eqref{eq:condodet} for any $t>0$ and \eqref{eq:condode}  are equivalent.
\end{remark}

\begin{corollary}\label{cor:1}
    If the operator $A$ is $L^{\Phi}$-dissipative, then
    \eqref{eq:posit} holds  for almost every $x\in(a,b)$ and for any 
    $\la\in\C^{m}$.
\end{corollary}
\begin{proof}
In the proof of Theorem \ref{th:1} we have already seen that
inequality \eqref{eq:condode} implies  \eqref{eq:posit}. 
\end{proof} 
 
 We have also

 \begin{lemma}
The operator $A$ is strict $L^{\Phi}$-dissipative if and only if
there exists $\kappa >0$ such that
\begin{gather*}
  \int_{a}^{b}\Big( \Re \lan \A v',v'\ran-\La^{2}(|v|)\, |v|^{-4} \Re \lan \A v,v\ran 
	  (\Re \lan v,v'\ran)^{2} \cr
	  +\La(|v|)\, |v|^{-2} \Re (\lan \A v,v'\ran 
	  -\lan \A v',v\ran) \Re \lan v,v'\ran
	  \Big)\, dx  \geq \kappa  \int_{a}^{b} |v'|^{2}dx,
\end{gather*}
for any $v\in [\Cspt^1(\Om)]^m$.
\end{lemma}
 \begin{proof}
 It is a particular case of Lemma \ref{le:lemma2}. 
 \end{proof}

 \begin{corollary}\label{cor:2}
    Suppose
    \begin{equation}\label{eq:condL}
\La^{2}_{\infty}=\sup_{t>0}\La^{2}(t) < 1.
\end{equation}
The operator $A$ is strict $L^{\Phi}$-dissipative if and only if
there exists $\kappa >0$ such that $A-\kappa I (d^{2}/dx^{2})$ is
$L^{\Phi}$-dissipative.
\end{corollary}
\begin{proof}
It is a particular case of 
 \cite[Corollary 1]{CM2022}.
\end{proof}

Following the ideas of \cite{CM2006} we prove
\begin{theorem}\label{th:2}
Let us assume condition \eqref{eq:condL}.
  There exists $\kappa >0$ such that $A-\kappa I (d^{2}/dx^{2})$ is
$L^{\Phi}$-dissipative if and only if
\begin{equation}\label{eq:infP}
\essinf_{\genfrac{}{}{0 pt}{}{(x,\la,\om) \in(a,b)\times \C^m\times \C^m}{|\la|=|\om|=1}}
P(x,\la,\om) > 0,
\end{equation}
where
\begin{gather*}
P(x,\la,\om) =  \Re \lan \A(x) \la,\la\ran - \La_{\infty}^{2}\Re\lan \A(x)\om,\om\ran (\Re \lan\la,\om\ran)^{2}
	\\
	+
	\La_{\infty}\Re(\lan \A(x)\om,\la\ran -\lan \A(x)\la,\om\ran)
	\Re \lan \la,\om\ran \, .
\end{gather*}
\end{theorem}
\begin{proof}
Thanks to Theorem \ref{th:1}, $A-\kappa I (d^{2}/dx^{2})$ is
$L^{\Phi}$-dissipative if and only if 
\begin{equation}\label{eq:condconP}
P(x,\la,\om) - \kappa (|\la|^{2} -  \La_{\infty}^{2} (\Re \lan\la,\om\ran)^{2}) \geq 0
\end{equation}
 for almost every $x\in(a,b)$ and for any 
    $\la,\om\in\C^{m}$, $|\om|=1$. Since
    \begin{equation}\label{eq:vecchia54}
|\la|^{2} -  \La_{\infty}^{2} (\Re \lan\la,\om\ran)^{2} \geq (1-\La_{\infty}^{2})|\la|^{2} > 0
\end{equation}
    for any $\la\neq 0$, there exists $\kappa >0$ such that \eqref{eq:condconP} holds if and only if
    \begin{equation}\label{eq:vecchia55}
\essinf_{\genfrac{}{}{0 pt}{}{(x,\la,\om) \in(a,b)\times \C^m\times \C^m}{|\la|=|\om|=1}}
\frac{P(x,\la,\om)}{1 -  \La_{\infty}^{2} (\Re \lan\la,\om\ran)^{2}} >0.
\end{equation}
On the other hand inequality \eqref{eq:vecchia54} leads to
$$
P(x,\la,\om) \leq \frac{P(x,\la,\om)}{1 -  \La_{\infty}^{2} (\Re \lan\la,\om\ran)^{2}}  \leq
(1 -  \La_{\infty}^{2})^{-1} P(x,\la,\om)
$$
for almost every $x\in(a,b)$ and for any 
    $\la,\om\in\C^{m}$, $|\la|=|\om|=1$.  This shows that \eqref{eq:infP} and \eqref{eq:vecchia55}
    are equivalent and this concludes the proof.
\end{proof}

\begin{corollary} \label{co:equivstrict}
Let us assume condition \eqref{eq:condL}. 
The operator $A$ is strict $L^{\Phi}$-dissipative if and only if there exists $\kappa> 0$
such that
    \begin{gather*}
  \Re \lan \A(x) \la,\la\ran - \La_{\infty}^{2}\Re\lan \A(x)\om,\om\ran (\Re \lan\la,\om\ran)^{2}
	\cr
	+
	\La_{\infty}\Re(\lan \A(x)\om,\la\ran -\lan \A(x)\la,\om\ran)
	\Re \lan \la,\om\ran   \geq \kappa\, |\la|^{2}
    \end{gather*}
    for almost every $x\in(a,b)$ and for any 
    $\la,\om\in\C^{m}$, $|\om|=1$.
\end{corollary}
\begin{proof}
It follows immediately from Corollary \ref{cor:2} and Theorem \ref{th:2},
because \eqref{eq:infP} means that there exists $\kappa>0$ such that
$$
P(x,\la,\om)\geq \kappa \, |\la|^{2}
$$
for any $\la,\, \om\in \C^{m}$, $|\om|=1$ and for almost every $x\in\Om$.
\end{proof}


\section{Real coefficient operators}\label{sec:real}

This section is devoted to real coefficient operators. We shall give different 
necessary and sufficient conditions for the functional dissipative which are expressed in terms of the eigenvalues of the matrix $\A$. 

\begin{theorem}\label{th:4}
Let $\A$ be a real matrix $\{a_{hk}\}$ with $h,k=1,\ldots,m$.
 Let us suppose $\A=\A^{t}$ and  $\A > 0$
 (in the sense $\lan\A(x)\xi,\xi\ran > 0$, for almost every 
 $x\in(a,b)$ and for any $\xi\in\R^{m}\setminus\{0\}$).
 The operator $A$ is 
    $L^{\Phi}$-dissipative if and only if
    $$
   \La^{2}_{\infty} (\mu_{1}(x)
	+\mu_{m}(x))^{2} \leq 4\, \mu_{1}(x)\mu_{m}(x)
    $$
    almost everywhere, 
    where $\mu_{1}(x)$ and $\mu_{m}(x)$ are the smallest and the largest 
    eigenvalues of the matrix $\A(x)$ respectively. In the particular 
    case $m=2$, this condition is equivalent to
    $$
    \La^{2}_{\infty} (\tr \A(x))^{2}  \leq 4\, \det \A(x) 
    $$
    almost everywhere.
\end{theorem}
\begin{proof}
In view of Theorem \ref{th:1} we have the $L^{\Phi}$-dissipativity of $A$ if and only if 
\eqref{eq:condode} holds for almost every $x\in(a,b)$ and for any 
    $\la,\om\in\C^{m}$, $|\om|=1$. Reasoning as in the proof of \cite[Theorem 5, p.255]{CM2006},
    we see that in the present case this condition is equivalent to
   $$
        \lan \A(x) \xi,\xi\ran-
        \La^{2}_{\infty}  \lan \A(x) \om,\om\ran (\lan \xi,\om\ran)^{2} \geq 0
 $$
    for almost every $x\in(a,b)$ and for 
    any $\xi, \om\in\R^{m}$, $|\om|=1$. 
    As in  \cite[Theorem 5, p.255]{CM2006}, this inequality is satisfied if and only if
    $$
     \La^{2}_{\infty}(\mu_{h}\om_h^2) (\mu_k^{-1}\om_k^2) \leq 1
    $$
for any $\om\in \R^m$, $|\om|=1$, i.e., if and only if (see \cite[Lemma 4, p.253]{CM2006})
$$
\La^{2}_{\infty}\frac{(\mu_{1}
	+\mu_{m})^{2}}{4\mu_{1}\mu_{m}} \leq 1.
$$
The result for $m=2$ follows from the identities
\begin{equation}
    \mu_{1}(x)\mu_{2}(x)=\det \A(x) , \quad \mu_{1}(x)+\mu_{2}(x)=\tr \A(x) .
    \label{identeigen}
\end{equation}
\end{proof}

In the rest of the Section we assume condition \eqref{eq:condL}.

\begin{corollary}\label{cor:3}
     Let $\A$ be a real and symmetric matrix. 
    Denote by $\mu_{1}(x)$ and $\mu_{m}(x)$ the 
       smallest and the largest eigenvalues of $\A(x)$ respectively.
       There exists $\kappa >0$ such that 
       $A-\kappa I(d^{2}/dx^{2})$ is 
       $L^{\Phi}$-dissipative if and only if
     \begin{equation}
	       \essinf_{x\in(a,b)}
	     \left[\left(1+\sqrt{1-\La_{\infty}^{2}}\right)\mu_{1}(x) - \left(1-\sqrt{1-\La_{\infty}^{2}} \right)\,
	     \mu_{m}(x)
	     \right]
		     >0.
	 \label{primacondmu}
     \end{equation}
   In the particular case  $m=2$, condition \eqref{primacondmu} is
   equivalent to
   \begin{equation} 
	   \essinf_{x\in(a,b)}\left[ \sqrt{1-\La_{\infty}^{2}} \tr \A(x) - 
	  \sqrt{(\tr\A(x))^{2}- 
	   4\det\A(x)} \right]
	   >0.
       \label{primacondmu2}
   \end{equation}
\end{corollary}
\begin{proof}
\textit{Necessity.} 
From inequality \eqref{eq:condconP} we deduce that $A- \kappa' I (d^2/dx^2)$ is $L^{\Phi}$-dissipative
for any $0<\kappa'<\kappa$. Moreover, 
by Corollary \ref{cor:1}, $\A(x)-\kappa'I > 0$ almost everywhere for any $0<\kappa'<\kappa$.
By Theorem \ref{th:4} we find that
 \begin{equation}\label{eq:69}
 \La^{2}_{\infty} (\mu_{1}(x)
	+\mu_{m}(x )- 2 \kappa')^{2} \leq 4\, (\mu_{1}(x)-\kappa')(\mu_{m}(x)-\kappa')
\end{equation}
    almost everywhere. By observing that $4\mu_1\mu_m=(\mu_1+\mu_m)^2-(\mu_1-\mu_m)^2$,
    the last inequality can be written as
\begin{equation}\label{eq:70}
(1- \La^{2}_{\infty}) (\mu_{1}(x)
	+\mu_{m}(x )- 2 \kappa')^{2} -(\mu_1(x)-\mu_m(x))^2 \geq 0
\end{equation}
    almost everywhere. Since this holds for any $\kappa'<\kappa$, we have that $\kappa$ is less than or equal to the smallest root of
    the left hand side  of \eqref{eq:70}, i.e.
    \begin{equation}\label{eq:primacondk}
 \kappa \leq \frac{1}{2} \left(\left(1+1/\sqrt{1-\La^{2}_{\infty}}\right)\mu_1(x) + \left(1-1/\sqrt{1-\La^{2}_{\infty}}\right)\mu_m(x)\right)
\end{equation}
     almost everywhere and \eqref{primacondmu} is proved.
     
     \textit{Sufficiency.} Let $\kappa'$ be such that
     \begin{equation}\label{eq:newk}
 0<\kappa' <  \essinf_{x\in(a,b)}  \frac{1}{2} \left(\left(1+1/\sqrt{1-\La^{2}_{\infty}}\right)\mu_1(x) + \left(1-1/\sqrt{1-\La^{2}_{\infty}}\right)\mu_m(x)\right).
\end{equation}
     
     Since  $(1-1/\sqrt{1-\La^{2}_{\infty}})\mu_m(x)\leq (1-1/\sqrt{1-\La^{2}_{\infty}} )\mu_1(x)$ and then
     \begin{gather*}
     \left(1+1/\sqrt{1-\La^{2}_{\infty}} \right) \mu_1(x) +
     \left(1-1/\sqrt{1-\La^{2}_{\infty}} \right)\mu_m(x)
     \leq 2\mu_1(x),
\end{gather*}
condition \eqref{eq:newk} shows that $\A(x)-\kappa' I>0$ almost everywhere. Moreover
inequality \eqref{eq:newk} implies that $\kappa'$ satisfies  \eqref{eq:69}.
The result follows from Theorem \ref{th:4}.

If $m=2$  the equivalence between \eqref{primacondmu} and \eqref{primacondmu2} 
follows from the identities \eqref{identeigen}.
\end{proof}

Under an additional assumption on the matrix $\A$ we have also

 \begin{corollary}
     		Let $\A$ be a real and symmetric matrix. 
     		Suppose $\A > 0$ almost everywhere.
     		Denote by $\mu_{1}(x)$ and $\mu_{m}(x)$ the 
     		   smallest and the largest eigenvalues of $\A(x)$ respectively.
     		 If  there exists $\kappa >0$ such that 
     		   $A-\kappa I(d^{2}/dx^{2})$ is 
     		   $L^{\Phi}$-dissipative, then
     		 \begin{equation}
     			   \essinf_{x\in(a,b)}
     			 \left[
     			 \mu_{1}(x)\mu_{m}(x) - \frac{\La^{2}_{\infty}}{2}
     			 (\mu_{1}(x)+\mu_{m}(x))^{2}
     			 \right]
     				 >0.
     		     \label{primacondmun}
     		 \end{equation}
		 If, in addition,  there exists
		$C$ such that 
		\begin{equation}
		    \lan\A(x)\xi,\xi\ran \leq C|\xi|^{2}
		    \label{condC}
		\end{equation}
		for almost every $x\in(a,b)$ and for any $\xi\in\R^{m}$,
		the converse is also true.
        In the particular case  $m=2$ condition \eqref{primacondmun} 
        is equivalent to
        $$
             \essinf_{x\in(a,b)}\left[\det \A(x)
     	       -\frac{\La^{2}_{\infty}}{2}
     	    (\tr\A(x))^{2}\right] > 0 .
         $$
        \end{corollary}
 \begin{proof}
         \textit{Necessity.} By the proof of Corollary \ref{cor:3}, 
        inequality \eqref{eq:primacondk} holds. On the other hand we have
	 	  \begin{gather*}
\left(\left(1+1/\sqrt{1-\La^{2}_{\infty}}\right)\mu_1(x) + \left(1-1/\sqrt{1-\La^{2}_{\infty}}\right)\mu_m(x)\right) \\
\leq \left(\left(1-1/\sqrt{1-\La^{2}_{\infty}}\right)\mu_1(x) + \left(1+1/\sqrt{1-\La^{2}_{\infty}}\right)\mu_m(x)\right)
\end{gather*}
	 	    and then
		    \begin{gather*}
4k^{2}\leq \left(\left(1+1/\sqrt{1-\La^{2}_{\infty}}\right)\mu_1(x) + \left(1-1/\sqrt{1-\La^{2}_{\infty}}\right)\mu_m(x)\right) \\
\times 
\left(\left(1-1/\sqrt{1-\La^{2}_{\infty}}\right)\mu_1(x) + \left(1+1/\sqrt{1-\La^{2}_{\infty}}\right)\mu_m(x)\right)
\end{gather*}
almost everywhere.
	 	This inequality can be written as
$$
	         (1-\La^{2}_{\infty}) k^2 \leq 
	              \mu_{1}(x)\mu_{m}(x) - \frac{\La^{2}_{\infty}}{2}
	                           (\mu_{1}(x)+\mu_{m}(x))^{2}
$$
	 	and \eqref{primacondmun} is proved.
		
		\textit{Sufficiency.} Let $h>0$ such that
	 $$
	(1-\La^{2}_{\infty}) h \leq 
	              \mu_{1}(x)\mu_{m}(x) - \frac{\La^{2}_{\infty}}{2}
	                           (\mu_{1}(x)+\mu_{m}(x))^{2}
	 $$
	 almost everywhere, i.e.
	 		    \begin{gather*}
4 h \leq \left(\left(1+1/\sqrt{1-\La^{2}_{\infty}}\right)\mu_1(x) + \left(1-1/\sqrt{1-\La^{2}_{\infty}}\right)\mu_m(x)\right) \\
\times 
\left(\left(1-1/\sqrt{1-\La^{2}_{\infty}}\right)\mu_1(x) + \left(1+1/\sqrt{1-\La^{2}_{\infty}}\right)\mu_m(x)\right)
\end{gather*}
	 almost everywhere. Since $\mu_{1}(x)> 0$, we have also
	\begin{gather*}
 \left(1-1/\sqrt{1-\La^{2}_{\infty}}\right)\mu_1(x) + \left(1+1/\sqrt{1-\La^{2}_{\infty}}\right)\mu_m(x) \\
 \leq \left(1+1/\sqrt{1-\La^{2}_{\infty}}\right)\mu_m(x)\end{gather*} 
	 and then
	 \begin{gather*}
4h 
\leq \left(\left(1+1/\sqrt{1-\La^{2}_{\infty}}\right)^{2} \mu_1(x) + \left(1-1/(1-\La^{2}_{\infty})\right)\mu_m(x)\right) \esssup_{y\in(a,b)}\mu_m(y)
\end{gather*}
	 almost everywhere. By \eqref{condC}, $\esssup\mu_{m}$ is finite and
	 by \eqref{primacondmun} it is greater than zero. Then
	 \eqref{primacondmu} holds and Corollary \ref{cor:3} gives the result.
     \end{proof}

\begin{remark} As showed in \cite[Remark 2, p.258]{CM2006}, condition \eqref{condC}
cannot be omitted. 
\end{remark}

\section{Functional dissipativity  for some systems of partial differential equations}
\label{sec:PDE}

In this Section we consider  the particular class of matrix operators defined as
 \begin{equation}\label{eq:47}
    Au = \de_{h}(\A^{h}(x)\de_{h}u)
\end{equation}
where $\A^{h}(x)=\{a_{ij}^{h}(x)\}$  ($i,j=1,\ldots,m$)
are matrices
with complex locally integrable entries defined in a domain 
$\Om\subset\R^{n}$ ($h=1,\ldots,n$).
We  give necessary and sufficient condition for its functional dissipativity.

We start by considering the strict $L^{\Phi}$-dissipativity of operator \eqref{eq:47} and
by explicitly writing the next result, which is  particular case of Lemma \ref{le:lemma2}.
\begin{lemma}\label{le:lemma2bis}
The operator  \eqref{eq:47} is strict $L^{\Phi}$-dissipative if and only if there exists $\kappa>0$
such that
\begin{equation}\label{eq:cond1bis}
\begin{gathered}
\Re \int_{\Om} \Big(\lan \A^{h} \de_h	 v, \de_h v\ran 
- \La^{2}(|v|)\, |v|^{-4} \lan \A^{h} v,v\ran (\Re \lan v, \de_{h} v\ran)^{2} \\
+\La(|v|)\, |v|^{-2} \lan \left(\A^{h}- (\A^{h})^*\right)v, \de_h v \ran \Re \lan v, \de_{h} v\ran  
\Big) dx  
\geq  \kappa \int_{\Om} |\nabla v|^{2}dx
\end{gathered}
\end{equation}
for any $v\in (\Cspt^{1}(\Om))^{m}$.
\end{lemma}

By $y_h$ we denote the $(n-1)$-dimensional
vector $(x_{1},\ldots,x_{h-1}, x_{h+1}, \ldots ,x_{n})$
and we set
$\om(y_{h})=\{ x_{h}\in\R\ |\ x\in\Om\}$.

\begin{lemma}\label{lemma:5bis}
Let us assume condition \eqref{eq:condL}. 
    The operator \eqref{eq:47} is strict $L^{\Phi}$-dissipative if and only if
    the ordinary differential operators \eqref{eq:ODops}
    are uniformly strict $L^{\Phi}$-dissipative in $\om(y_{h})$
    for almost every $y_{h}\in \R^{n-1}$
    ($h=1,\ldots,n$).  This condition is void if 
    $\om(y_{h})=\emptyset$.
\end{lemma}
\begin{proof}
\textit{Sufficiency.}
Suppose $r\geq 0$. The strict $L^{\Phi}$-dissipative can be written as
\begin{gather*}
    \Re  \sum_{h=1}^{n}\int_{\R^{n-1}}dy_{h}\int_{\om(y_{h})}
    \lan \A^{h}\de_{h}u, \de_{h}(\vf(|u|) u))\ran dx_{h} \\
    \geq \kappa  \sum_{h=1}^{n}\int_{\R^{n-1}}dy_{h}\int_{\om(y_{h})} |\de_{h}
    (\sqrt{\vf(|u|)}\, u)|^{2} dx_{h}
\end{gather*}
for any  $u\in (\Cspt^{1}(\Om))^{m}$.
 By assumption,
there exists $\kappa >0$ such that
 \begin{gather*}
 \Re \int_{\om(y_{h})}
    \lan \A^{h}(x)v'(x_{h}), 
    (\vf(|v(x_h)|)v(x_{h}))'\ran dx_{h} \\
    \geq \kappa \int_{\om(y_{h})}
    | ((\sqrt{\vf(|v(x_{h})|)}\, v(x_{h}))'|^{2} dx_{h}
\end{gather*}
   for almost every $y_{h}\in \R^{n-1}$ and
   for any $v\in(\Cspt^{1}(\om(y_{h})))^{m}$,
   provided $\om(y_{h})
   \neq \emptyset$ ($h=1,\ldots,n$).   This implies
    \begin{equation}\label{eq:repl}
 \Re \sum_{h=1}^{n}\int_{\Om}\lan \A^{h}(x)\de_{h}u,
        \de_{h}(\vf(|u|) u)\ran dx \geq \kappa
        \int_{\Om}  |\nabla
    (\sqrt{\vf(|u|)}\, u)|^{2} dx
\end{equation}
for any $u\in (\Cspt^{1}(\Om))^{m}$, and the strict $L^{\Phi}$-dissipativity is proved.
The proof for $-1<r<0$ runs in the same way. We have just to replace 
\eqref{eq:repl} by 
\begin{equation}\label{eq:sameway}
\Re  \sum_{h=1}^{n} \int_{\Om} \lan (\A^{h})^{*} \de_{h} w, \de_{h}(\psi(|w|)\, w)\ran\, dx \geq 
\kappa \int_{\Om} |\nabla (\sqrt{\psi(|w|)}\, w)|^{2} dx
\end{equation}
for any $w\in  (\Cspt^{1}(\Om))^{m}$.

\textit{Necessity.} 
Let $r\geq 0$. First suppose that $\A^h$ are constant matrices and $\Om=\R^n$.
Fix $1\leq k\leq n$. Let $\al\in(\Cspt^{1}(\R))^{m}$ and
	$\be$ be a real valued scalar function in $\Cspt^{1}(\R^{n-1})$. Consider
	$$
	u_{\ep}(x)=\al(x_{k}/\ep)\, \be(y_{k}). 
	$$

By assumption, there exists $\kappa>0$ such that
\begin{equation}\label{eq:ass}
\Re  \sum_{h=1}^{n} \int_{\Om} \lan \A^{h} \de_{h} u_{\ep}, \de_{h}(\vf(|u_{\ep}|)\, u_{\ep})\ran\, dx \geq \kappa 
        \int_{\Om}  |\nabla
    (\sqrt{\vf(|u_{\ep}|)}\, u_{\ep})|^{2} dx
.
\end{equation}

Keeping in mind that $\be$ is a scalar real valued function, if $h\neq k$ we  have
\begin{gather*}
 \de_{h}(\vf(|u_{\ep}|)\, u_{\ep}) =
 g(|u_{\ep}|)\, \de_{h} u_{\ep}\, ,
\end{gather*}
where $g(s)= (s \vf(s))'$, and
\begin{gather*}
\de_{h}(\sqrt{\vf(|u_{\ep}|)}\, u_{\ep}) =
\widetilde{g}(|u_{\ep}|)\, \de_{h}  u_{\ep}\, ,
\end{gather*}
where $\widetilde{g}(s)= (s \sqrt{\vf(s)})'$.

If $h=k$ we can write
$$
\de_{k}(\vf(|u_{\ep}|)\, u_{\ep}) = \ep^{-1}\, \de_{t}\ga(x_{k}/\ep,\be(y_{k}))\, \be(y_{k}),
$$
where $\ga(t,s)=\vf(| \al(t)\, s|) \,\al(t)$, and
$$
\de_{k}(\sqrt{\vf(|u_{\ep}|)}\, u_{\ep}) = \ep^{-1}\, \de_{t}\widetilde{\ga}(x_{k}/\ep,\be(y_{k}))\, \be(y_{k}),
$$
where $\widetilde{\ga}(t,s)=\sqrt{\vf(| \al(t)\, s|)} \,\al(t)$.

Therefore inequality \eqref{eq:ass} can be written as

\begin{gather*}
	 \ep^{-2} \Re \int_{\R^{n-1}}\be^{2}(y_{k})dy_{k}
	 \int_{\R}\lan \A^{k}
	 \al'(x_{k}/\ep), 
	 \de_{t}\ga(x_{k}/\ep,\be(y_{k}))\ran 
	 \, dx_{k}
	 \\
	 + 
	 \sum_{\genfrac{}{}{0 pt}{}{h=1}{h\neq k}}^{n}
	 \Re \int_{\R^{n-1}} (\de_{h}\be(y_{k}))^{2}dy_{k}
	  \int_{\R}  \lan \A^{h}\al(x_{k}/\ep),\al(x_{k}/\ep)\ran \,g[|\al(x_{k}/\ep)\, \be(y_{k})|] dx_{k}
	 \\
	 \geq \kappa\,
	 \ep^{-2}\int_{\R^{n-1}}\be^{2}(y_{k})dy_{k}
	 \int_{\R} | \de_{t}\widetilde{\ga}(x_{k}/\ep,\be(y_{k}))|^{2} \,
	 dx_{k}
	 \\
	 + 
	 \kappa  \sum_{\genfrac{}{}{0 pt}{}{h=1}{h\neq k}}^{n}
	 \int_{\R^{n-1}} (\de_{h}\be(y_{k}))^{2}dy_{k}
	  \int_{\R} |\al(x_{k}/\ep)|^{2}  \left(\widetilde{g}\left(\left| \al(x_{k}/\ep) \be(y_{k}\right| \right)\right)^{2}dx_{k} ,
	  \end{gather*}
	  i.e.
	 \begin{gather*}
	 \ep^{-1}\Re \int_{\R^{n-1}}\be^{2}(y_{k})dy_{k}
	 \int_{\R}\lan \A^{k}
	 \al'(t), 
	 \de_{t}\ga(t,\be(y_{k}))\ran 
	 \, dt
		  \cr
		  + \ep   \sum_{\genfrac{}{}{0 pt}{}{h=1}{h\neq k}}^{n}
		 \Re  \int_{\R^{n-1}}(\de_{h}\be(y_{k}))^{2}dy_{k}
		  \int_{\R} \Re \lan \A^{h}\al(t),\al(t)\ran 
		  \, g[|\al(t)\, \be(y_{k})|]  dt
		  \\
		  \geq \kappa\, \ep^{-1}
		 \int_{\R^{n-1}}\be^{2}(y_{k})dy_{k}
	 \int_{\R} | \de_{t}\widetilde{\ga}(t,\be(y_{k}))|^{2} \,
	 dt
	 	 \\
	 +  \kappa\, \ep 
	 \sum_{\genfrac{}{}{0 pt}{}{h=1}{h\neq k}}^{n}
	 \int_{\R^{n-1}} (\de_{h}\be(y_{k}))^{2}dy_{k}
	  \int_{\R} |\al(t)|^{2}  \left(\widetilde{g}\left(\left| \al(t) \be(y_{k}\right| \right)\right)^{2}dt\, .
\end{gather*}
    Letting $\ep \to 0^{+}$  we get
    \begin{equation}\label{eq:newdis}
\begin{gathered}
\Re \int_{\R^{n-1}}\be^{2}(y_{k})dy_{k}
	 \int_{\R}\lan \A^{k}
	 \al'(t), 
	 \de_{t}\ga(t,\be(y_{k}))\ran 
	 \, dt \\
		   \geq \kappa
		   \int_{\R^{n-1}}\be^{2}(y_{k})dy_{k}
	 \int_{\R} | \de_{t}\widetilde{\ga}(t,\be(y_{k}))|^{2} \,
	 dt\,  .
\end{gathered}
\end{equation}

We claim that from the arbitrariness of $\be$ it follows
\begin{equation}\label{eq:conuno}
\Re		 \int_{\R} \lan \A^{k}
	 \al'(t), 
	 \de_{t}\ga(t,1))\ran 
	 \, dt \\
		   \geq \kappa
	  \int_{\R} | \de_{t}\widetilde{\ga}(t,1)|^{2} \,
	 dt
		    .
\end{equation}

We prove \eqref{eq:conuno} by contradiction. 
Set
$$
\Phi(s)= 
		  \int_{\R} \left[ \Re \lan \A^{k}
	 \al'(t), 
	 \de_{t}\ga(t,s))\ran 
		    - \kappa\,		 
	  | \de_{t}\widetilde{\ga}(t,s)|^{2} \right]dt
$$
and suppose $\Phi(1)<0$. 
Take $\be_{\ep}\in \Cspt^{1}(\R^{n-1})$ such that $0\leq \be_{\ep} \leq 1$, 
$\be_{\ep}(y_{k})=1$ on the unit ball $B_{1}(0)$ and
$\spt \be_{\ep} \subset B_{1+\ep}(0)$.
Since $\Phi(1)<0$ and keeping in mind \eqref{eq:newdis}, we can write
\begin{gather*}
0\leq \int_{\R^{n-1}} \be_{\ep}^{2}(y_{k}) \Phi( \be_{\ep}(y_{k})) dy_{k}\\
  =\int_{\R^{n-1}} \be_{\ep}^{2}(y_{k}) [ \Phi( \be_{\ep}(y_{k})) - \Phi(1)] dy_{k} +
  \Phi(1)  \int_{\R^{n-1}} \be_{\ep}^{2}(y_{k})  dy_{k}\\
  \leq \int_{B_{1+\ep}(0)\setminus B_{1}(0)}\be_{\ep}^{2}(y_{k}) [ \Phi( \be_{\ep}(y_{k})) - \Phi(1)] dy_{k} +
  \Phi(1)\, | B_{1}(0)|\, .
\end{gather*}
Letting $\ep\to 0^{+}$, we get $0\leq  \Phi(1)\, | B_{1}(0)| <0$. This is a contradiction and \eqref{eq:conuno} is proved. We then have
$$
\Re   \int_{\R}\lan \A^{k}
		  \al'(t), 
		 (|\vf(|\al(t)|) \al(t))' \ran 
		  \, dt \geq \kappa
		  \int_{\R} |  (\sqrt{\vf(|\al(t)|)}\, \al(t))'  |^{2} dt
$$
for any $\al\in(\Cspt^{1}(\R))^{m}$. This means that $A(y_{k})$ is strict $L^{\Phi}$-dissipative.

If $\A^{h}$ are not necessarily constant and defined in $\Om$,
we observe that, thanks to Lemma \ref{le:lemma2bis},  $A$ is strict $L^{\Phi}$-dissipative if and
only if condition \eqref{eq:cond1bis} holds.

Consider
    $$
    v(x)=w((x-x_{0})/\ep)
    $$
    where $x_{0}\in\Om$ is a fixed point, $w\in(\Cspt^{1}(\R^{n}))^{m}$,
    $\spt  w \subset B_{1}(0)$ and
    $0<\ep<\dist(x_{0},\de\Om)$.
    Putting this particular $v$ in \eqref{eq:cond1bis} we get
    \begin{gather*}
	\Re \int_{\Om} \Big(\lan \A^{h}(x) \de_h	 w((x-x_{0})/\ep), \de_h w((x-x_{0})/\ep)\ran 
 \\ 
 -
\La^{2}(|w((x-x_{0})/\ep)|)\, |w((x-x_{0})/\ep)|^{-4} \lan \A^{h}(x) w((x-x_{0})/\ep),w((x-x_{0})/\ep)\ran \\
 \times (\Re \lan w((x-x_{0})/\ep), \de_{h} w((x-x_{0})/\ep)\ran)^{2}
 \\
 +
\La(|w((x-x_{0})/\ep)|) 
 |w((x-x_{0})/\ep)|^{-2} \\ \times
 \lan (\A^{h}(x)- (\A^{h})^*(x))w((x-x_{0})/\ep), \de_h w((x-x_{0})/\ep) \ran 
 \\  \times
\Re \lan w((x-x_{0})/\ep), \de_{h} w((x-x_{0})/\ep)\ran 
\Big) dx \\
\geq \kappa \int_{\Om} |\nabla w((x-x_{0})/\ep)|^{2} dx
\end{gather*}
and then
\begin{gather*}
 \Re \int_{\Om} \Big(\lan \A^{h}(x_{0}+\ep y) \de_h w(y), \de_h w(y)\ran 
 \\  
  - 
\La^{2}(|w(y)|)\, |w(y)|^{-4} \lan \A^{h}(x_{0}+\ep y) w(y),w(y)\ran
(\Re \lan w(y), \de_{h} w(y)\ran)^{2}\\ 
 +
\La(|w((y)|) 
 |w(y)|^{-2}
 \lan (\A^{h}(x_{0}+\ep y)- (\A^{h})^*(x_{0}+\ep y))w(y), \de_h w((y) \ran 
 \\  \times
\Re \lan w(y), \de_{h} w(y)\ran 
\Big) dy   
\geq \kappa \int_{\Om} |\nabla w(y)|^{2} dy
\, .
\end{gather*}

	Letting $\ep\to 0^{+}$, we obtain
	\begin{gather*}
\Re \int_{\R^{n}}\Big( \lan \A^{h}(x_{0}) \de_{h}w,\de_{h}w\ran
		 - \La^{2}(|w|)\|w|^{-4} \lan \A^{h}(x_{0}) w,w\ran (\Re \lan 
		 w,\de_{h}w\ran)^{2} \cr
		 +  \La^{2}(|w|)\ |w|^{-2}(\lan \A^{h}(x_{0}) w,\de_{h}w\ran
		- \lan \A^{h}(x_{0}) \de_{h}w,w\ran)\Re \lan w,\de_{h} w\ran
		  \Big) dy \\
		 \geq \kappa \int_{\Om} |\nabla w(y)|^{2} dy
		 \end{gather*}
	for almost every $x_{0}\in\Om$. 
	
	Because of the arbitrariness of 
	$w \in(\Cspt^{1}(\R^{n}))^{m}$, Lemma \ref{le:lemma2bis} shows that the constant
	coefficient operator $\de_{h}(\A^{h}(x_{0})\de_{h})$ is 
	strict $L^{\Phi}$-dissipative. From what has already been proved, 
	the ordinary differential operators $(\A^{h}(x_{0})v')'$ are strict
	$L^{\Phi}$-dissipative ($h=1,\ldots,n$).
	
	By Corollary \ref{co:equivstrict} we have
	\begin{equation}\label{condnewstrict}
	\begin{gathered}
 \Re \lan \A^{h}(x_{0}) \la,\la\ran -\La_{\infty}^{2}\Re\lan 
	  \A^{h}(x_{0})\om,\om\ran (\Re \lan\la,\om\ran)^{2}
	   \cr
	   + \La_{\infty}
	    \Re(\lan \A^{h}(x_{0})\om,\la\ran -
	    \lan \A^{h}(x_{0})\la,\om\ran)
	       \Re \lan \la,\om\ran   \geq \kappa'\ |\la|^{2}
	       \end{gathered}
\end{equation}
	for any $\la,\om\in\C^{m}$, $|\om|=1$, $h=1,\ldots,n$, where $\kappa'=(1-\La_{\infty}^{2})\kappa$.

	Reasoning as in \cite[p.261]{CM2006} we find
	that, for almost every $y_{h}\in \R^{n-1}$, we
	have
	\begin{gather*}
 \Re \lan \A^{h}(x) \la,\la\ran -\La_{\infty}^{2}\Re\lan 
	  \A^{h}(x)\om,\om\ran (\Re \lan\la,\om\ran)^{2}
	   \cr
	   + \La_{\infty}
	    \Re(\lan \A^{h}(x)\om,\la\ran -
	    \lan \A^{h}(x)\la,\om\ran)
	       \Re \lan \la,\om\ran   \geq \kappa'\ |\la|^{2}
	       \end{gather*}
	for almost every $x_{h}\in \om(y_{h})$ and 
	for any $\la,\om\in\C^{m}$, $|\om|=1$, provided
	$\om(y_{h}) \neq \emptyset$.
	The conclusion follows from Corollary \ref{co:equivstrict}.
    
     If $-1<r<0$ the proof runs in the same way, recalling \eqref{eq:sameway}.
\end{proof}

We are now in a position to prove the necessary and sufficient algebraic condition
for the strict $L^{\Phi}$-dissipativity of $A$.

\begin{theorem}\label{thx:3strict}
    Let us assume condition \eqref{eq:condL}.  The operator \eqref{eq:47} is strict $L^{\Phi}$-dis\-si\-pa\-tive 
    if and only if there exists $\kappa'>0$ such that
    \eqref{condnewstrict} holds for almost every $x_{0}\in\Om$ and for any $\la,\om\in\C^{m}$, $|\om|=1$,
    $h=1,\ldots,n$.
\end{theorem}
\begin{proof}
  \textit{Necessity.} This has been already proved in the necessity
    part of  Lemma \ref{lemma:5bis}. 
    
    \textit{Sufficiency.} 
   We know that if \eqref{condnewstrict} holds for
    almost every $x_{0}\in\Om$ and
    for any $\la,\om\in\C^{m}$, $|\om|=1$,  
     the ordinary differential
    operators $A(y_{h})$ are uniformly strict $L^{\Phi}$-dissipative 
    for almost every $y_{h}\in \R^{n-1}$, provided
    $\om(y_{h})\neq \emptyset$
    ($h=1,\ldots,n$). By Lemma \ref{lemma:5bis},  $A$ is 
    strict $L^{\Phi}$-dissipative.    
\end{proof}

Similar results hold for the $L^{\Phi}$-dissipativity. 

\begin{lemma}
    The operator \eqref{eq:47} is $L^{\Phi}$-dissipative if and only if
    the ordinary differential operators
   \begin{equation}\label{eq:ODops}
A(y_{h})[u(x_{h})]=d(\A^{h}(x)du/dx_{h})/dx_{h}
\end{equation}
    are $L^{\Phi}$-dissipative in $\om(y_{h})$
    for almost every $y_{h}\in \R^{n-1}$
    ($h=1,\ldots,n$). This condition is void if 
    $\om(y_{h})=\emptyset$.
\end{lemma}
\begin{proof}
The proof runs as the proof of Lemma \ref{lemma:5bis}, in which we set $\kappa=0$.
We note that here we do not need to assume condition  \eqref{eq:condL}.
\end{proof}

\begin{theorem}\label{thx:3}
    The operator \eqref{eq:47} is $L^{\Phi}$-dissipative 
    if and only if
    $$
	\begin{gathered}
 \Re \lan \A^{h}(x_{0}) \la,\la\ran -\La_{\infty}^{2}\Re\lan 
	  \A^{h}(x_{0})\om,\om\ran (\Re \lan\la,\om\ran)^{2}
	   \cr
	   + \La_{\infty}
	    \Re(\lan \A^{h}(x_{0})\om,\la\ran -
	    \lan \A^{h}(x_{0})\la,\om\ran)
	       \Re \lan \la,\om\ran   \geq 0
	       \end{gathered}
$$
 holds for almost every $x_{0}\in\Om$ and for any $\la,\om\in\C^{m}$, $|\om|=1$,
    $h=1,\ldots,n$.
\end{theorem}
\begin{proof}
 The proof is similar to the proof of Theorem \ref{thx:3strict}.
\end{proof}

In the case of a real coefficient operator \eqref{eq:47}, we have also

\begin{theorem}
    Let $A$ be the operator \eqref{eq:47}, where
$\A^{h}$ are real matrices $\{a^{h}_{ij}\}$ with $i,j=1,\ldots,m$.
 Let us suppose $\A^{h}=(\A^{h})^{t}$ and  $\A^{h}> 0$
 ($h=1,\ldots,n$).
 The operator $A$ is 
    $L^{\Phi}$-dissipative if and only if
   \begin{equation}
	 \La^{2}_{\infty} (\mu_{1}^{h}(x)
	   +\mu_{m}^{h}(x))^{2} \leq 
	  4\, \mu_{1}^{h}(x)\,
	   \mu_{m}^{h}(x)
       \label{ineigen}
   \end{equation}
    for almost every $x\in\Om$, $h=1,\ldots,n$,  
    where $\mu_{1}^{h}(x)$ and $\mu_{m}^{h}(x)$ are the 
    smallest and the largest 
    eigenvalues of the matrix $\A^{h}(x)$ respectively. 
    In the particular 
    case $m=2$ this condition is equivalent to
    $$
     \La^{2}_{\infty}(\tr \A^{h}(x))^{2}  \leq 4\, \det 
    \A^{h}(x) 
    $$
    for almost every $x\in\Om$, $h=1,\ldots,n$.
\end{theorem}

\begin{proof}
    By Theorem \ref{thx:3}, $A$ is $L^{\Phi}$-dissipative if and only 
    if 
    $$
    \lan \A^{h}(x) \la,\la\ran-
     \La^{2}_{\infty}\lan \A^{h}(x) \om,\om\ran (\Re\lan \la,\om\ran)^{2} \geq 0
 $$
  for almost every $x\in\Om$, for any $\la,\om\in\C^{m}$, $|\om|=1$, 
  $h=1,\ldots,n$.  
    The proof of Theorem \ref{th:4} shows that
  these conditions are equivalent to \eqref{ineigen}.
\end{proof}

\section{Functional ellipticity}\label{sec:ellipt}

In \cite{DLP} \textsc{Dindo\v{s}, Li} and \textsc{Pipher} have introduced 
different  notions of $p$-ellipticity, establishing relationships between them.
Here we extend these definitions in the frame of functional ellipticity, introduced in 
\cite{CM2021} for scalar operators.

Coming back to the general operator \eqref{eq:A}, we say that the tensor $\{a^{hk}_{ij}(x)\}$ satisfies the strong $\Phi$-ellipticity condition
if there exists $\kappa >0$ such that
\begin{equation}\label{eq:strongphi}
\begin{gathered}
\Re \lan \A^{hk}(x)\xi_k, \xi_h \ran -
\La^{2}(t) \Re \lan (\A^{hk}(x)-(\A^{kh})^*(x))\om,\xi_h \ran \Re \lan \om, \xi_k\ran 
 \\
 + \La(t) \lan  \A^{hk}(x)\om, \om \ran  \Re \lan \om, \xi_k \ran  \Re \lan \om, \xi_h\ran
 \geq \kappa |\xi|^2
\end{gathered}
\end{equation}
for any $\xi_{h}, \om\in \C^m$, $|\om|=1$, $t>0$ and for almost every $x\in\Om$.

We say that the tensor $\{a^{hk}_{ij}(x)\}$ satisfies the integral $\Phi$-ellipticity condition
if there exists $\kappa>0$ such that condition \eqref{eq:cond1str} holds for any
$v \in [\Cspt^{1}(\Om)]^m$.  We note that if there are no lower order terms, as in
the case we are considering here, the concepts of strong dissipativity and integral
ellipticity are equivalent, thanks to Lemma \ref{le:lemma2}. This is not the case if there are lower order terms. The $\Phi$-ellipticity 
is still given by  \eqref{eq:cond1str}, while the formula for $\Phi$-dissipativity has to be changed,
taking into account the lower order terms.

It is trivial that strong $\Phi$-ellipticity implies integral $\Phi$-ellipticity. Indeed, 
\eqref{eq:strongphi} implies that the integrand in  \eqref{eq:cond1str} is non-negative almost
everywhere.

We say that the tensor $\{a^{hk}_{ij}(x)\}$  satisfies the Legendre-Hadamard $\Phi$-ellipticity condition
(or weak $\Phi$-ellipticity condition) if there exixts $\kappa>0$ such that 
\begin{equation}\label{eq:weakphi}
\begin{gathered}
  \Re  \lan (\A^{hk}(x)q_{h}q_{k})\la,\la\ran  -\La^{2}(t) \Re\lan (\A^{hk}(x)q_{h}q_{k})\om,\om\ran (\Re \lan 
    \la,\om\ran)^{2}\cr
    + \La(t)\Re(\lan (\A^{hk}(x)q_{h}q_{k})\om,\la\ran
	-\lan (\A^{hk}(x)q_{h}q_{k})\la,\om\ran) 
	\Re \lan \la,\om\ran 
	\cr\geq \kappa\, |q|^{2} |\la|^{2}
\end{gathered}
\end{equation}
	for any $q\in\R^{n}$, $\la,\, \om\in \C^{m}$, $|\om|=1$, $t>0$ and for almost every $x\in\Om$.

If $\vf(t)=t^{p-2}$ and then $\La(t)=-(1-2/p)$, conditions \eqref{eq:strongphi}, \eqref{eq:cond1str} and \eqref{eq:weakphi}
coincide with  (17), (20) and (31) of \cite{DLP}, respectively. 

We remark that, if condition \eqref{eq:condL} is satisfied, then inequalities
\eqref{eq:strongphi} and \eqref{eq:weakphi} for any $t>0$ are equivalent to 
\begin{gather*}
\Re \lan \A^{hk}(x)\xi_k, \xi_h \ran -
\La_{\infty}^{2} \Re \lan (\A^{hk}(x)-(\A^{kh})^*(x))\om,\xi_h \ran \Re \lan \om, \xi_k\ran 
 \\
 + \La_{\infty} \lan  \A^{hk}(x)\om, \om \ran  \Re \lan \om, \xi_k \ran  \Re \lan \om, \xi_h\ran
 \geq \kappa |\xi|^2
\end{gather*}
and
\begin{equation}\label{eq:newweakcond}
\begin{gathered}
  \Re  \lan (\A^{hk}(x)q_{h}q_{k})\la,\la\ran  -\La_{\infty}^{2} \Re\lan (\A^{hk}(x)q_{h}q_{k})\om,\om\ran (\Re \lan 
    \la,\om\ran)^{2}\cr
    + \La_{\infty}\Re(\lan (\A^{hk}(x)q_{h}q_{k})\om,\la\ran
	-\lan (\A^{hk}(x)q_{h}q_{k})\la,\om\ran) 
	\Re \lan \la,\om\ran 
	\cr\geq \kappa\, |q|^{2} |\la|^{2},
\end{gathered}
\end{equation}
respectively (see Remark \ref{rem:equiv}).

We can now rephrase some of our results as follows. 

\begin{theorem}
    Let $n=1$ and $A$ be the operator \eqref{eq:Aord}. Assume \eqref{eq:condL} holds. The 
   following statements are equivalent:
   \renewcommand{\labelenumi}{(\alph{enumi})}
\renewcommand{\theenumi}{(\alph{enumi})}
   \begin{enumerate}
\item\label{eq:1} the operator $A$ is strict $L^{\Phi}$-dissipative;
\item\label{eq:2} there exists $\kappa>0$ such that $A-k I (d^2/dx^2)$ is $L^{\Phi}$-dissipative;
\item\label{eq:3} the matrix $\{a_{ij}(x)\}$ satisfies the strong $\Phi$-ellipticity condition;
\item\label{eq:4} the matrix $\{a_{ij}(x)\}$ satisfies the integral  $\Phi$-ellipticity condition;
\item\label{eq:5} the matrix $\{a_{ij}(x)\}$ satisfies the weak $\Phi$-ellipticity condition.
\end{enumerate}
\end{theorem}
\begin{proof}
The equivalence between \ref{eq:1} and \ref{eq:2} is given by Corollary \ref{cor:2}. 
The equivalence between  \ref{eq:1} and  \ref{eq:5} was proved in Corollary \ref{co:equivstrict}.
 It is also clear that,
 when $n=1$, \eqref{eq:strongphi} and \eqref{eq:weakphi} coincide and then
  \ref{eq:3} and  \ref{eq:5} are equivalent. Finally, we have already remarked that, if the operator has no lower order terms,  conditions \ref{eq:1} and \ref{eq:4} are equivalent.
\end{proof}

 An analogous result holds in any dimension for the operator \eqref{eq:47}.

 \begin{theorem}
    Let $n\geq 2$ and $A$ be the operator \eqref{eq:47}. Assume \eqref{eq:condL} holds. The 
   following statements are equivalent:
   \renewcommand{\labelenumi}{(\alph{enumi})}
\renewcommand{\theenumi}{(\alph{enumi})}
   \begin{enumerate}
\item\label{eq:1b} the operator $A$ is strict $L^{\Phi}$-dissipative;
\item\label{eq:2b} there exists $\kappa>0$ such that $A-k \Delta$ is $L^{\Phi}$-dissipative;
\item\label{eq:3b} the matrix $\{a_{ij}(x)\}$ satisfies the integral  $\Phi$-ellipticity condition;
\item\label{eq:4b} the matrix $\{a_{ij}(x)\}$ satisfies the weak $\Phi$-ellipticity condition.
\end{enumerate}
Moreover, if the matrix $\{a_{ij}(x)\}$ satisfies the strong $\Phi$-ellipticity condition, then 
\ref{eq:1b}-\ref{eq:4b} hold.
\end{theorem}
\begin{proof}
The equivalence between \ref{eq:1b} and \ref{eq:2b} was proved in \cite[Corollary 1]{CM2022}
for the more general operator \eqref{eq:A}. The equivalence between \ref{eq:1b} and \ref{eq:3b}
is given by Lemma \ref{le:lemma2bis}. To prove the equivalence between \ref{eq:1b} 
and \ref{eq:4b}, we first note that condition \eqref{eq:newweakcond} in the present case
is
\begin{gather*}
  \sum_{h=1}^{n} \Big(\Re \lan (\A^{h}(x)q_{h}^{2})\la,\la\ran  -\La_{\infty}^{2} \Re\lan (\A^{h}(x)q_{h}^{2})\om,\om\ran (\Re \lan 
    \la,\om\ran)^{2}\cr
    + \La_{\infty}\Re(\lan (\A^{h}(x)q_{h}^{2})\om,\la\ran
	-\lan (\A^{h}(x)q_{h}^{2})\la,\om\ran) 
	\Re \lan \la,\om\ran \Big)
	\\
	\geq \kappa\, |q|^{2} |\la|^{2},
\end{gather*}
for any $q\in\R^{n}$, $\la,\, \om\in \C^{m}$, $|\om|=1$, $t>0$ and for almost every $x\in\Om$.
Thanks to the arbitrariness of the vector $q$, this is equivalent to the $n$ inequalities
\begin{gather*}
 \Re \lan \A^{h}(x) \la,\la\ran -\La_{\infty}^{2}\Re\lan 
	  \A^{h}(x)\om,\om\ran (\Re \lan\la,\om\ran)^{2}
	   \cr
	   + \La_{\infty}
	    \Re(\lan \A^{h}(x)\om,\la\ran -
	    \lan \A^{h}(x)\la,\om\ran)
	       \Re \lan \la,\om\ran   \geq \kappa \, |\la|^{2}
	       \end{gather*}
	for any $\la,\om\in\C^{m}$, $|\om|=1$, $h=1,\ldots,n$. 
	Therefore Theorem \ref{thx:3strict} gives the equivalence between \ref{eq:1b} 
and \ref{eq:4b}.
Finally, we know that the strong $\Phi$-ellipticity  implies
the integral $\Phi$-ellipticity and the proof is complete.
\end{proof}


\begin{thebibliography}{10}
\providecommand{\selectlanguage}[1]{\relax}
\expandafter\ifx\csname urlstyle\endcsname\relax
  \providecommand{\doi}[1]{doi: \discretionary{}{}{}#1}\else
  \providecommand{\doi}{doi: \discretionary{}{}{}\begingroup
  \urlstyle{rm}\Url}\fi

\bibitem{CD20201}
\textsc{Carbonaro, A.} and \textsc{Dragi\v{c}evi\'{c}, O.} (2020), Bilinear
  embedding for divergence-form operators with complex coefficients on
  irregular domains.
\newblock \emph{Calc. Var. Partial Differential Equations}, \textbf{59}(3),
  Paper No. 104, 36 pp.
\newblock \doi{10.1007/s00526-020-01751-3}.

\bibitem{CD20202}
\textsc{Carbonaro, A.} and \textsc{Dragi\v{c}evi\'{c}, O.} (2020), Convexity of
  power functions and bilinear embedding for divergence-form operators with
  complex coefficients.
\newblock \emph{J. Eur. Math. Soc. (JEMS)}, \textbf{22}(10), 3175--3221.
\newblock \doi{10.4171/jems/984}.

\bibitem{CM2005}
\textsc{Cialdea, A.} and \textsc{Maz'ya, V.} (2005), Criterion for the
  {$L^p$}-dissipativity of second order differential operators with complex
  coefficients.
\newblock \emph{J. Math. Pures Appl. (9)}, \textbf{84}(8), 1067--1100.
\newblock \doi{10.1016/j.matpur.2005.02.003}.

\bibitem{CM2006}
\textsc{Cialdea, A.} and \textsc{Maz'ya, V.} (2006), Criteria for the
  {$L^p$}-dissipativity of systems of second order differential equations.
\newblock \emph{Ric. Mat.}, \textbf{55}(2), 233--265.
\newblock \doi{10.1007/s11587-006-0014-x}.

\bibitem{CM2013}
\textsc{Cialdea, A.} and \textsc{Maz'ya, V.} (2013), {$L^p$}-dissipativity of
  the {L}am\'{e} operator.
\newblock \emph{Mem. Differ. Equ. Math. Phys.}, \textbf{60}, 111--133.

\bibitem{CMbook}
\textsc{Cialdea, A.} and \textsc{Maz'ya, V.} (2014), \emph{Semi-bounded
  differential operators, contractive semigroups and beyond}, \emph{Operator
  Theory: Advances and Applications}, vol. 243.
\newblock Birkh\"{a}user/Springer, Cham.

\bibitem{CM2018}
\textsc{Cialdea, A.} and \textsc{Maz'ya, V.} (2018), The {$L^p$}-dissipativity
  of first order partial differential operators.
\newblock \emph{Complex Var. Elliptic Equ.}, \textbf{63}(7-8), 945--960.
\newblock \doi{10.1080/17476933.2017.1321638}.

\bibitem{CM2021}
\textsc{Cialdea, A.} and \textsc{Maz'ya, V.} (2021), Criterion for the
  functional dissipativity of second order differential operators with complex
  coefficients.
\newblock \emph{Nonlinear Anal.}, \textbf{206}, 112215.
\newblock \doi{10.1016/j.na.2020.112215}.

\bibitem{CM2022}
\textsc{Cialdea, A.} and \textsc{Maz’ya, V.} (2022), Criterion for the
  functional dissipativity of the Lamé operator.
\newblock \emph{Eur. J. Mech. A Solids}, 104522.
\newblock \doi{10.1016/j.euromechsol.2022.104522}.

\bibitem{CM20222}
\textsc{Cialdea, A.} and \textsc{Maz’ya, V.} (2022), A survey of functional
  and $L^p$-dissipativity theory.
\newblock \emph{Bull. Math. Sci.}, 2230003.
\newblock \doi{10.1142/S1664360722300031}.

\bibitem{DLP}
\textsc{Dindo\v{s}, M.}, \textsc{Li, J.} and \textsc{Pipher, J.} (2021), The
  {$p$}-ellipticity condition for second order elliptic systems and
  applications to the {L}am\'{e} and homogenization problems.
\newblock \emph{J. Differential Equations}, \textbf{302}, 367--405.
\newblock \doi{10.1016/j.jde.2021.08.036}.

\bibitem{DP20191}
\textsc{Dindo\v{s}, M.} and \textsc{Pipher, J.} (2019), Perturbation theory for
  solutions to second order elliptic operators with complex coefficients and
  the {$L^p$} {D}irichlet problem.
\newblock \emph{Acta Math. Sin. (Engl. Ser.)}, \textbf{35}(6), 749--770.
\newblock \doi{10.1007/s10114-019-8214-y}.

\bibitem{DP20192}
\textsc{Dindo\v{s}, M.} and \textsc{Pipher, J.} (2019), Regularity theory for
  solutions to second order elliptic operators with complex coefficients and
  the {$L^p$} {D}irichlet problem.
\newblock \emph{Adv. Math.}, \textbf{341}, 255--298.
\newblock \doi{10.1016/j.aim.2018.07.035}.

\bibitem{DP20201}
\textsc{Dindo\v{s}, M.} and \textsc{Pipher, J.} (2020), Boundary value problems
  for second-order elliptic operators with complex coefficients.
\newblock \emph{Anal. PDE}, \textbf{13}(6), 1897--1938.
\newblock \doi{10.2140/apde.2020.13.1897}.

\bibitem{DP20202}
\textsc{Dindo\v{s}, M.} and \textsc{Pipher, J.} (2020), Extrapolation of the
  {D}irichlet problem for elliptic equations with complex coefficients.
\newblock \emph{J. Funct. Anal.}, \textbf{279}(7), 108693, 20.
\newblock \doi{10.1016/j.jfa.2020.108693}.

\bibitem{egert}
\textsc{Egert, M.} (2020), On {$p$}-elliptic divergence form operators and
  holomorphic semigroups.
\newblock \emph{J. Evol. Equ.}, \textbf{20}(3), 705--724.
\newblock \doi{10.1007/s00028-019-00537-1}.

\end{thebibliography}
\end{document}